\documentclass[11pt,reqno]{amsart}

\usepackage[T1]{fontenc}
\usepackage{amsmath,amssymb,amsthm,mathtools}
\usepackage{enumitem,geometry}
\usepackage{microtype}
\usepackage{cite}
\usepackage{xcolor}
\usepackage[colorlinks=true,
            linkcolor=blue!50!black,
            citecolor=blue!50!black,
            urlcolor=blue!50!black]{hyperref}
\allowdisplaybreaks[4]
\numberwithin{equation}{section}

\newtheorem{theorem}{Theorem}[section]
\newtheorem{proposition}[theorem]{Proposition}
\newtheorem{lemma}[theorem]{Lemma}
\newtheorem{corollary}[theorem]{Corollary}

\theoremstyle{definition}

\theoremstyle{remark}
\newtheorem{remark}[theorem]{Remark}
\newtheorem{example}[theorem]{Example}

\newcommand{\C}{\mathbb C}
\newcommand{\Sym}{\operatorname{Sym}}
\newcommand{\End}{\operatorname{End}}
\newcommand{\Hom}{\operatorname{Hom}}
\newcommand{\Ind}{\operatorname{Ind}}
\newcommand{\Res}{\operatorname{Res}}
\newcommand{\Stab}{\operatorname{Stab}}
\newcommand{\rad}{\operatorname{rad}}

\newcommand{\Triv}{\mathbf 1}
\newcommand{\sgn}{\operatorname{sgn}}
\newcommand{\GL}{\operatorname{GL}}

\newcommand{\cF}{\mathcal F}

\title{Invariant Jacobian and the Center of a Homogeneous Form}
\author[]{Yue Hu}
\address[]{Y.H.: Beijing University of Posts and Telecommunications, School of Mathematical Sciences, Beijing, China}
\email[]{huyue@bupt.edu.cn}

\date{September 2, 2026}
\hypersetup{
  pdftitle={Invariant Jacobian and the Center of a Homogeneous Form},
  pdfauthor={Yue Hu}
}

\begin{document}
\begin{abstract}
We study homogeneous forms with invariant Jacobian under 
group actions.  For finite-dimensional irreducible complex representations, we show
that the canonical decomposition defined by the center of the form
is either trivial or gives a system of imprimitivity, which  determines the module structure of Jacobian $J(f)$.  We apply the theory to symmetric groups and classify homogeneous forms with invariant Jacobian on the natural permutation module.
\end{abstract}

\maketitle

\section{Introduction}\label{sec:introduction}

Let $V$ be a finite-dimensional complex vector space and let
$f\in\Sym^d(V^*)$ be a homogeneous form of degree $d\geq3$.  The \emph{Jacobian}  of $f$  is
\begin{equation}\label{eq:intro-J}
  J(f)=\langle D_vf:v\in V\rangle
  \subseteq\Sym^{d-1}(V^*),
\end{equation}
where $D_v$ denotes directional differentiation. If a group $G$ acts linearly on $V$, then $J(g\cdot f)=g\cdot J(f)$. Thus a relative invariant $f$ (that is $\C f$ is invariant) always has $G$-invariant Jacobian.
The question is to what extent the converse holds:
\begin{equation}\label{eq:main-question}
  J(f)\text{ is $G$-invariant}
  \quad\Longrightarrow\quad
  \C f\text{ is $G$-invariant}.
\end{equation}

This problem goes back to work of Yau on isolated hypersurface
singularities and the Lie algebras associated with their moduli
algebras \cite{MatherYau,YauMemoir}. Motivated by a conjecture of Yau, Kempf studied invariant Jacobians for rational representations of connected semisimple algebraic groups.  He proved
\eqref{eq:main-question} when the representation is irreducible
or the projective hypersurface $f=0$ is smooth,  and
also showed that an invariant Jacobian admits an invariant form with the same Jacobian \cite{Kempf}.  Xi used Kempf's theorem to prove
Yau's highest-weight conjecture and then turned to invariant
Jacobians for general groups \cite{Xi}. He
asked for the $G$-module structure of $J(f)$, the existence of a
relative invariant with the same Jacobian as $f$, and a
description of all forms having the same Jacobian as $f$, see \cite[Section~3.5]{Xi}.

Our approach is based on the center $Z(f)$ of a higher degree form.
For a nondegenerate form $f$ (equivalently,
$\dim J(f)=\dim V$), $Z(f)$ is a finite-dimensional commutative
algebra, whose primitive idempotents determine the canonical
decompositions
\[
  V=V_1\oplus\cdots\oplus V_r,
  \qquad
  f=f_1+\cdots+f_r,
\]
see \cite{Harrison,ORyanShapiro,HLYZ}.  If $V$ is irreducible and $J(f)$ is $G$-invariant, then $f$ is
nondegenerate and $\rad Z(f)=0$.  The induced action of $G$ on
$Z(f)$ therefore permutes the canonical summands
$V_1,\ldots,V_r$, and irreducibility forces this action to be
transitive.  Thus either $r=1$, or
\[
  V=V_1\oplus\cdots\oplus V_r
\]
is a nontrivial system of imprimitivity for $V$. This is the structural link between invariant Jacobians and
imprimitivity.  It reduces the problem to one canonical summand and
its stabilizer.  Our main structure theorem shows that the canonical decomposition determines the $G$-module structure of $J(f)$ and all homogeneous forms having the same
Jacobian.  Thus Theorem~\ref{thm:jacobian-structure} gives a uniform
answer to Xi's questions for arbitrary finite-dimensional irreducible
complex representations. In particular, if $V$ admits no nontrivial system of imprimitivity,
that is, if $V$ is primitive,  we obtain:

\begin{theorem}\label{thm:primitive}
Let $G$ be a group and let $V$ be a finite-dimensional primitive
irreducible complex $G$-module.  If
\[
  0\ne f\in\Sym^d(V^*),\qquad d\geq3,
\]
and $J(f)$ is $G$-invariant, then $\C f$ is $G$-invariant.
\end{theorem}

For connected semisimple algebraic groups, this recovers Kempf's
theorem for irreducible modules. Moreover, for finite groups,  we show that primitivity
exactly characterizes the irreducible representations for which
\eqref{eq:main-question} holds:

\begin{theorem}\label{thm:finite-converse}
Let $G$ be a finite group and let $V$ be a finite-dimensional irreducible complex $G$-module.
Then the following are equivalent:
\begin{enumerate}[label=\textup{(\roman*)},leftmargin=0.9cm]
\item $V$ is primitive;
\item for every $d\geq3$ and every
      $0\ne f\in\Sym^d(V^*)$, $G$-invariance of $J(f)$ implies
      $G$-invariance of $\C f$.
\end{enumerate}
\end{theorem}

Finally, we apply the general theory to symmetric groups $S_n$, with $n\geq 5$. Using the known classification of systems of imprimitivity for irreducible
Specht modules \cite{DjokovicMalzan}, we obtain a
concrete description of invariant Jacobians on irreducible
$S_n$-modules, see Theorem \ref{thm:Sn-classification}. We then consider the natural permutation module $P_n$, which is reducible and lies outside the irreducible theory
above.   Using the special submodule structure of $P_n$, we  determine the possible canonical decompositions and
obtain all forms with invariant Jacobians, yielding a complete answer to Xi's questions for $P_n$.

The paper is organized as follows.
Section~\ref{sec:center} recalls the center of a form and develops
its relation with invariant Jacobians.
Section~\ref{sec:imprimitivity} proves the main theorems stated  above.
Section~\ref{sec:Sn} applies the theory to symmetric groups.

\section{Center and Invariant Jacobian}\label{sec:center}

\subsection{Center of a form}

We begin by recalling the center theory of higher degree forms.
The basic notions go back to Harrison \cite{Harrison}. We use the
coordinate-free and Hessian descriptions in the form developed by Huang--Lu--Ye--Zhang \cite{HLYZ},  see also
\cite{Proszynski,ORyanShapiro}.  We also record the relation between the center and the Jacobian fibre. The latter appears in different forms in Kempf \cite{Kempf},
Huang--Lu--Ye--Zhang \cite{HLYZ}, and, most explicitly, in the
symmetrizer-group formulation of Hwang \cite{Hwang}.  The purpose of
this section is to collect these facts in a form suited to the group
actions used later.

 Let $V$ be
a finite-dimensional complex vector space and $d\geq3$. A form $f\in\Sym^d(V^*)$ determines a symmetric $d$-linear form $\Theta_f:V^d\longrightarrow\C$ by polarization:
\begin{equation}\label{eq:polarization}
  \Theta_f(v_1,\ldots,v_d)
  =
  \frac{1}{d!}
  \left.
  \frac{\partial^d}{\partial t_1\cdots\partial t_d}
  f(t_1v_1+\cdots+t_dv_d)
  \right|_{t_1=\cdots=t_d=0},
\end{equation}
in which $\Theta_f(v,\ldots,v)=f(v)$. Conversely, if
$\Theta:V^d\to\C$ is a symmetric $d$-linear form, then $f_\Theta(v):=\Theta(v,\ldots,v)$ is a form in $\Sym^d(V^*)$. The assignments
\[
  f\longmapsto\Theta_f,
  \qquad
  \Theta\longmapsto f_\Theta
\]
are inverse to each other, giving a canonical identification between
$\Sym^d(V^*)$ and the space of symmetric $d$-linear forms on $V$. We shall use  the following identities
\begin{gather}
  D_vf(x)
  =
  d\,\Theta_f(v,x,\ldots,x), \label{eq:directional-polarization}
\\ 
D_uD_vf(x)
  =
  d(d-1)\Theta_f(u,v,x,\ldots,x)\label{eq:polarization-2}.
\end{gather}
The form $f$ is \emph{nondegenerate} if no variable can be removed by
an invertible linear change of variables.  This is equivalent to
\[
  f^\perp:=\{v\in V:D_vf=0\}=0,
\]
or in terms of the polarization,
\begin{equation}\label{eq:nondegenerate}
  \Theta_f(v,v_2,\ldots,v_d)=0
  \text{ for all }v_2,\ldots,v_d\in V
  \quad\Longrightarrow\quad v=0.
\end{equation}
If $V=V_1\oplus\cdots\oplus V_r$ and $\Theta_f(v_1,\ldots,v_d)=0$ unless all arguments belong to the
same $V_i$, then the decomposition is orthogonal and
\begin{equation}\label{eq:direct-sum}
  f=f_1+\cdots+f_r,
  \qquad f_i\in\Sym^d(V_i^*).
\end{equation}
Here each $f_i$ is regarded as a form on $V$ via the projection
$V\to V_i$.  Thus an orthogonal decomposition of a form is precisely
a decomposition into a sum of forms in disjoint sets of variables.

Let $H_f=(\partial^2f/\partial x_i\partial x_j)$ be the Hessian
matrix.  The \emph{center} of $f$ is
\begin{equation}\label{eq:center-hessian}
  Z(f)=\{X\in\End(V):(H_fX)^{\mathsf T}=H_fX\}.
\end{equation}
Equivalently,
\begin{equation}\label{eq:center-polar}
\begin{aligned}
 Z(f)=\{X\in\End(V):{}&
 \Theta_f(Xv_1,v_2,v_3,\ldots,v_d)\\
 &{}=\Theta_f(v_1,Xv_2,v_3,\ldots,v_d)
 \text{ for all }v_i\in V\}.
\end{aligned}
\end{equation}
This is Harrison's center.  It also
appears in \cite[Section~5]{Kempf}.   The following are classical results of Harrison \cite{Harrison}.  We use the
formulation of \cite[Proposition~2.1]{HLYZ}, see also
\cite{ORyanShapiro}.

\begin{proposition}[Harrison]\label{prop:HLYZ-center}
Let $f$ be a nondegenerate form of degree $d\geq3$.
\begin{enumerate}[label=\textup{(\roman*)},leftmargin=0.9cm]
\item $Z(f)$ is a finite-dimensional commutative $\C$-algebra.
\item Direct-sum decompositions of $f$ are in bijection with complete
sets of pairwise orthogonal idempotents of $Z(f)$.
\item The decomposition of $f$ into indecomposable summands is unique
up to permutation and equivalence of the summands.
\end{enumerate}
\end{proposition}

Let $e_1,\ldots,e_r$ be the primitive idempotents of $Z(f)$.  The
resulting decompositions
\begin{equation}\label{eq:canonical-decomposition}
  V=V_1\oplus\cdots\oplus V_r,
  \qquad V_i=e_iV,
  \qquad f=f_1+\cdots + f_r
\end{equation}
will be  called the \emph{canonical decomposition} associated to  $f$.  Here $V_1,\dots,V_r$ will be called the canonical blocks of $f$, 
$f_i=f_{e_i}$ is the restriction of $f$ to $V_i$, inflated to $V$ by
the projection $e_i$.  The Artinian decomposition of the center is
\begin{equation}\label{eq:local-center}
  Z(f)=Z_1\times\cdots\times Z_r,
  \qquad Z_i=e_iZ(f)=\C e_i\oplus N_i,
\end{equation}
where $N_i$ is the maximal ideal of the local algebra $Z_i$.  Thus
\[
  \rad Z(f)=N_1\oplus\cdots\oplus N_r.
\]

\subsection{The Jacobian fibre}

We call $\cF_f^{(d)}:=\{h\in\Sym^d(V^*):J(h)=J(f)\}$ the \emph{Jacobian fibre} through $f$.  In \cite{Hwang},   the Jacobian fibre was described using the symmetrizer group $Z(f)^{\times}$. We record this result  in a form suited to the equivariant arguments below. For $z\in Z(f)$ define $f_z\in\Sym^d(V^*)$ by
\begin{equation}\label{eq:fz}
  \Theta_{f_z}(v_1,\ldots,v_d)
  =
  \Theta_f(zv_1,v_2,\ldots,v_d).
\end{equation}
The center identity makes the right-hand side symmetric, so $f_z$ is
well defined.

\begin{proposition}[Jacobian fibre]\label{prop:Jacobian-fibre}
Let $f\in\Sym^d(V^*)$ be nondegenerate, with $d\geq3$.
\begin{enumerate}[label=\textup{(\roman*)},leftmargin=0.9cm]
\item For every $h\in\Sym^d(V^*)$ satisfying
$J(h)\subseteq J(f)$, there is a unique $z\in Z(f)$ such that $h=f_z$. Hence the assignment $z\mapsto f_z$ gives a vector-space
isomorphism
\begin{equation}\label{eq:center-linear-fibre}
  Z(f)\xrightarrow{\sim}
  \{h\in\Sym^d(V^*):J(h)\subseteq J(f)\}.
\end{equation}
\item $J(f_z)=J(f)$ if and only if $z\in Z(f)^\times$.
\smallskip

\item If $J(h)=J(f)$, then $Z(h)=Z(f)$. Consequently,
\begin{equation}\label{eq:homogeneous-fibre}
\begin{aligned}
  \cF_f^{(d)}&=\{f_z:z\in Z(f)^\times\}\\
  &=\left\{
     \sum_{i=1}^r\bigl(\lambda_i f_i+f_{n_i}\bigr):
     \lambda_i\in\C^\times,\ n_i\in N_i
    \right\}.
\end{aligned}
\end{equation}
\end{enumerate}

\end{proposition}

\begin{proof}
For a form $q\in\Sym^d(V^*)$, let
\begin{equation}\label{eq:nondegeneracyiso}
  \delta_q:V\longrightarrow J(q),\qquad
  v\longmapsto D_vq.
\end{equation}
Since $f$ is nondegenerate, $\delta_f$ is an isomorphism.

We first prove (i).  Let $h\in\Sym^d(V^*)$ satisfy
$J(h)\subseteq J(f)$.  Since $\delta_f$ is an isomorphism, there is a
unique linear map $z\in\End(V)$ such that
\begin{equation}\label{eq:Jacobian-fibre-z}
  D_vh=D_{zv}f\qquad(v\in V).
\end{equation}
We claim that $z\in Z(f)$.  For $u,v\in V$, equality of mixed
directional derivatives gives
\[
  D_uD_{zv}f
  =D_uD_vh
  =D_vD_uh
  =D_vD_{zu}f.
\]
Using \eqref{eq:polarization-2}, we obtain
\[ \Theta_f(zv,u,x,\ldots,x)
  =
  \Theta_f(zu,v,x,\ldots,x),\qquad \forall\ x\in V.\]
Polarizing the remaining $d-2$ variables through \eqref{eq:polarization}, we obtain
\[
  \Theta_f(zv,u,v_3,\ldots,v_d)
  =
  \Theta_f(zu,v,v_3,\ldots,v_d),\qquad \forall\ u,v,v_3,\ldots,v_d\in V.
\]
Since $\Theta_f$ is symmetric,
this is
\[
  \Theta_f(zv,u,v_3,\ldots,v_d)
  =
  \Theta_f(v,zu,v_3,\ldots,v_d),
\]
which gives $z\in Z(f)$. For $z\in Z(f)$, the definition of $f_z$ gives
\begin{equation}\label{eq:derivative-fz}
  D_vf_z=D_{zv}f\qquad(v\in V).
\end{equation}
Indeed, evaluating at $x\in V$,
\[
  D_vf_z(x)
  =d\,\Theta_{f_z}(v,x,\ldots,x)
  =d\,\Theta_f(zv,x,\ldots,x)
  =D_{zv}f(x).
\]
Thus \eqref{eq:Jacobian-fibre-z} and
\eqref{eq:derivative-fz} show that $D_v(h-f_z)=0$ for every $v\in V$. Hence $h-f_z$ is constant.  Since both $h$ and $f_z$ are homogeneous
of degree $d>0$, this constant is zero, and $h=f_z$.
This proves surjectivity of \eqref{eq:center-linear-fibre}.  It is also injective:
if $f_z=0$, then \eqref{eq:derivative-fz} gives
$D_{zv}f=0$ for every $v\in V$.  Since $f$ is nondegenerate,
$zv=0$ for every $v$, and so $z=0$.  Finally, the map $z\mapsto f_z$ is also linear.  This proves (i).

We next prove (ii).  By \eqref{eq:derivative-fz},
\[
  J(f_z)
  =\{D_{zv}f:v\in V\}
  =\delta_f(zV).
\]
Since $\delta_f:V\to J(f)$ is an isomorphism,
\[
  J(f_z)=J(f)
  \quad\Longleftrightarrow\quad
  zV=V
  \quad\Longleftrightarrow\quad
  z\in\GL(V).
\]
For $z\in Z(f)$, invertibility as an endomorphism of $V$ is
equivalent to invertibility in the algebra $Z(f)$. Therefore
\[
  J(f_z)=J(f)
  \quad\Longleftrightarrow\quad
  z\in Z(f)^\times.
\]

Finally, suppose that $J(h)=J(f)$.  By (i) and (ii), we may write
\[
  h=f_z
  \qquad\text{for some }z\in Z(f)^\times.
\]
Let $X\in\End(V)$.  Using the definition of $f_z$, we have
\[
\begin{aligned}
  X\in Z(h)
  &\Longleftrightarrow
  \Theta_h(Xv_1,v_2,v_3,\ldots,v_d)
  =
  \Theta_h(v_1,Xv_2,v_3,\ldots,v_d)
  \\
  &\Longleftrightarrow
  \Theta_f(zXv_1,v_2,v_3,\ldots,v_d)
  =
  \Theta_f(zv_1,Xv_2,v_3,\ldots,v_d).
\end{aligned}
\]
Since $z\in Z(f)$, $\Theta_f(zv_1,Xv_2,v_3,\ldots,v_d)
  =
  \Theta_f(v_1,zXv_2,v_3,\ldots,v_d)$, thus
  \[X\in Z(h)
  \quad\Longleftrightarrow\quad
  zX\in Z(f).
\]
As $z$ is a unit of $Z(f)$, this is equivalent to $X\in Z(f)$, which proves $Z(h)=Z(f)$. It remains to describe the Jacobian fibre.  By the Artinian decomposition
\eqref{eq:local-center}, every $z\in Z(f)$ has a unique expression
\[
  z=\sum_{i=1}^r(\lambda_i e_i+n_i),
  \qquad
  \lambda_i\in\C,\quad n_i\in N_i.
\]
Since $Z_i=\C e_i\oplus N_i$ is local with maximal ideal $N_i$, the
element $\lambda_i e_i+n_i$ is a unit of $Z_i$ if and only if
$\lambda_i\ne0$.  Hence
\[
  z\in Z(f)^\times
  \quad\Longleftrightarrow\quad
  \lambda_i\ne0\qquad(1\leq i\leq r).
\]
By the linearity of $z\mapsto f_z$ and
$f_{e_i}=f_i$, we have $f_z
  =
  \sum_{i=1}^r
  \bigl(\lambda_i f_i+f_{n_i}\bigr)$.
\end{proof}

\subsection{Invariant Jacobian}

We now add the group action. Let $G$ be a group acting linearly on $V$, and let it act on
$\Sym(V^*)$ by
\[
  (g\cdot h)(x)=h(g^{-1}x).
\]
We record the basic equivariance properties of directional
differentiation that will be used below. For $g\in G$, $v,x\in V$, and $h\in\Sym(V^*)$, the chain rule gives
\[
\begin{aligned}
  D_{gv}(g\cdot h)(x)
  &=
  \left.\frac{d}{dt}\right|_{t=0}
  h\bigl(g^{-1}(x+tgv)\bigr) \\
  &=
  \left.\frac{d}{dt}\right|_{t=0}
  h(g^{-1}x+tv) \\
  &=
  D_vh(g^{-1}x)
  =
  (g\cdot D_vh)(x).
\end{aligned}
\]
Hence
\begin{equation}\label{eq:directional-covariance}
  D_{gv}(g\cdot h)=g\cdot(D_vh),
\end{equation}
and therefore
\begin{equation}\label{eq:J-covariance}
  J(g\cdot f)=g\cdot J(f).
\end{equation}
In particular, if $\C f$ is $G$-invariant, then so is $J(f)$. For a subspace
$U\subseteq\Sym^m(V^*)$, put
\[
  U^\perp
  :=
  \{v\in V:D_vu=0\text{ for every }u\in U\}.
\]
Following Kempf \cite[Lemma~2 and Corollary~3]{Kempf}, for $d\geq3$,  one has $J(f)^\perp=f^\perp$. Hence, if $J(f)$ is $G$-invariant, then $f^\perp$ is a
$G$-submodule of $V$.  In particular:

\begin{lemma}\label{lem:essential-kernel}
 If $V$ is irreducible, $f\ne0$ and $J(f)$ is $G$-invariant, then $f$ is nondegenerate.
\end{lemma}

 We shall also use the following equivariance of differentiation,
which is used in \cite{Xi} to study the $G$-module structure of
invariant Jacobians.
\begin{lemma}\label{lem:Xi-differentiation}
The map
\begin{equation}\label{eq:Xi-map}
  \Sym(V^*)\otimes V\longrightarrow\Sym(V^*),
  \qquad
  h\otimes v\longmapsto D_vh,
\end{equation}
is a $G$-homomorphism.  If $\C f$ affords the character $\chi$, then
$J(f)$ is a quotient of $\chi\otimes V$. In particular, if $f$ is nondegenerate,
then
\[
  J(f)\cong\chi\otimes V.
\]
\end{lemma}

\begin{proof}
The first assertion is exactly
\eqref{eq:directional-covariance}.  Restricting
\eqref{eq:Xi-map} to $\C f\otimes V$ gives a surjective
$G$-homomorphism
\[
  \C f\otimes V\longrightarrow J(f).
\]
Its kernel is naturally identified with $f^\perp$.  The assertions
follow.
\end{proof}

\subsection{Action on the center and the canonical blocks}

If $J(f)$ is invariant under $G$, Kempf shows that \(G\) acts on \(Z(f)\) by conjugation,  cf. \cite[Lemma~7]{Kempf}. Combining this with the associative-algebra structure of \(Z(f)\), we will show that  if $V$ is irreducible, then $G$ acts on  the primitive idempotents, and hence on the canonical blocks  $V_1,\dots,V_r$, transitively, which leads to the imprimitivity of \(V\) when $r>1$.

\begin{lemma}\label{lem:center-covariance}
For $g\in G$ and $f\in\Sym^d(V^*)$,
\begin{equation}\label{eq:center-covariance}
  Z(g\cdot f)=gZ(f)g^{-1}.
\end{equation}
Moreover, for $z\in Z(f)$,
\begin{equation}\label{eq:fz-covariance}
  g\cdot f_z
  =
  (g\cdot f)_{gzg^{-1}}.
\end{equation}
\end{lemma}

\begin{proof}
Polarization is compatible with linear changes of variables:
\begin{equation*}
  \Theta_{g\cdot f}(v_1,\ldots,v_d)
  =
  \Theta_f(g^{-1}v_1,\ldots,g^{-1}v_d).
\end{equation*}
The defining identity for the center then gives
\[
  X\in Z(g\cdot f)
  \quad\Longleftrightarrow\quad
  g^{-1}Xg\in Z(f),
\]
which proves \eqref{eq:center-covariance}.  Similarly,
\[
\begin{aligned}
  \Theta_{g\cdot f_z}(v_1,\ldots,v_d)
  &=
  \Theta_f(zg^{-1}v_1,g^{-1}v_2,\ldots,g^{-1}v_d) \\
  &=
  \Theta_{g\cdot f}
  \bigl(gzg^{-1}v_1,v_2,\ldots,v_d\bigr),
\end{aligned}
\]
which proves \eqref{eq:fz-covariance}.
\end{proof}

Assume from now on that $f$ is nondegenerate and that $J(f)$ is
$G$-invariant.  By \eqref{eq:J-covariance},
\[
  J(g\cdot f)=J(f)
  \qquad(g\in G).
\]
Proposition~\ref{prop:Jacobian-fibre} therefore gives a unique
$a_g\in Z(f)^\times$ such that
\begin{equation}\label{eq:unit-ag}
  g\cdot f=f_{a_g}.
\end{equation}
The same proposition gives $Z(g\cdot f)=Z(f)$,  so Lemma~\ref{lem:center-covariance} gives 
\begin{equation}\label{eq:center-conjugation}
  gZ(f)g^{-1}=Z(f)
  \qquad(g\in G).
\end{equation}
Thus $G$ acts on $Z(f)$ by conjugation:
\[
  \alpha_g:Z(f)\longrightarrow Z(f),
  \qquad
  \alpha_g(z)=gzg^{-1}.
\]
Since each $\alpha_g$ is an algebra automorphism of $Z(f)$, it
preserves the Jacobson radical and permutes the primitive
idempotents.  Let $1=e_1+\cdots+e_r$ be the decomposition of the identity into primitive orthogonal
idempotents, and put $V_i=e_iV$. Then
\begin{equation}\label{eq:block-action}
  gV_i
  =
  \alpha_g(e_i)V.
\end{equation}
Thus $G$ permutes the canonical blocks of $f$.  Moreover,
\eqref{eq:fz-covariance} gives
\begin{equation}\label{eq:summand-action}
  g\cdot f_{e_i}
  =
  (g\cdot f)_{\alpha_g(e_i)}.
\end{equation}

The following argument parallels Kempf's radical argument in
\cite[Theorem~12]{Kempf}, but uses the associative-algebra structure
of $Z(f)$ and hence avoids the connectedness assumption.

\begin{lemma}\label{lem:radical-proper}
Assume that $f$ is nondegenerate,  $J(f)$ is $G$-invariant,
and put
\[
  N=\rad Z(f).
\]
Then $NV$ is a proper $G$-submodule of $V$. If, in addition, $V$ is irreducible, then
\begin{equation*}
  \rad Z(f)=0.
\end{equation*}
In this case $G$ acts transitively on the primitive idempotents of
$Z(f)$ and  the canonical blocks $V_1,\dots,V_r$.
\end{lemma}

\begin{proof}
The Jacobson radical is preserved by every algebra automorphism of
$Z(f)$, so
\[
  \alpha_g(N)=N
  \qquad(g\in G).
\]
Hence $NV$ is $G$-stable: if $n\in N$, $v\in V$, and $g\in G$, then
\[
  g(nv)
  =
  (gng^{-1})(gv)
  =
  \alpha_g(n)(gv)
  \in NV.
\]
Since $Z(f)$ is finite-dimensional, $N$ is nilpotent.  Thus
$N^m=0$ for some $m\geq1$. Therefore $NV$ is proper. Suppose now that $V$ is irreducible.  Then $NV=0$, which implies $N=0$.  Thus $Z(f)$ is a
finite-dimensional commutative semisimple algebra, so
\[
  Z(f)
  =
  \C e_1\times \cdots\times \C e_r.
\]
 Since $\alpha$ permutes the
primitive idempotents, for every $G$-orbit $\mathcal O$ on
$\{e_1,\ldots,e_r\}$ the idempotent $e_{\mathcal O}
  :=
  \sum_{e_i\in\mathcal O}e_i$
is fixed by $G$.  Hence
\[
  e_{\mathcal O}V
  =
  \bigoplus_{e_i\in\mathcal O}e_iV
\]
is a $G$-submodule of $V$.  By irreducibility, it is either $0$ or
$V$.  Since every $e_iV$ is nonzero, the first possibility is
excluded.  Hence every orbit is the whole set of primitive
idempotents, and the action is transitive.
\end{proof}

\section{Invariant Jacobian and imprimitivity of a module}
\label{sec:imprimitivity}

Let $G$ be a group and let $V$ be a finite-dimensional irreducible
complex $G$-module.  Recall that a \emph{system of imprimitivity} is
a decomposition
\begin{equation}\label{eq:imprimitivity-system}
  V=V_1\oplus\cdots\oplus V_r 
  \qquad (r>1)
\end{equation}
into nonzero subspaces which are permuted transitively by $G$.
The module is \emph{primitive} if it admits no such decomposition.
If $W=V_i$ is one of the blocks and  $K=\Stab_G(W)$, then the transitive action on the blocks identifies them with $G/K$, so
\[
  [G:K]=r
  \qquad\text{and}\qquad
  V\cong\Ind_K^G W.
\]
See \cite[Definition~5.7 and Theorem~5.9]{Isaacs}.

\subsection{Structure of invariant Jacobians}
\label{subsec:jacobian-structure}
\label{subsec:finite-fibres}

In \cite[Section~3.5]{Xi}, Xi asks, for a form $f$ with
$G$-invariant Jacobian $J(f)$, three related questions:
\begin{enumerate}[label=\textup{(\roman*)},leftmargin=0.9cm]
\item describe the $G$-module structure of $J(f)$;
\item determine whether there is a form $h$ with $J(h)=J(f)$ such
      that $\C h$ is $G$-invariant;
\item describe all forms having the same Jacobian as $f$.
\end{enumerate}
The center-theoretic results of Section~\ref{sec:center} give a
uniform answer to these questions for arbitrary finite-dimensional
irreducible complex representations. Let
\[
  V=V_1\oplus\cdots\oplus V_r,
  \qquad
  f=f_1+\cdots+f_r
\]
be the canonical decomposition determined by the primitive
idempotents $e_1,\ldots,e_r$ of $Z(f)$, where
\[
  V_i=e_iV,
  \qquad
  f_i=f_{e_i}.
\]
If $J(f)$ is $G$-invariant, then Lemma~\ref{lem:essential-kernel}
shows that $f$ is nondegenerate, while
Lemma~\ref{lem:radical-proper} gives
\[
  Z(f)=\C e_1\times\cdots\times\C e_r,
\]
and that $G$ permutes the $e_i$ and the $V_i$,
transitively.  Thus, if $r>1$, the canonical decomposition is a
system of imprimitivity for $V$. Fix one canonical block $W=V_1$ and put $ K=\Stab_G(W)$.

\begin{theorem}\label{thm:jacobian-structure}
Let $G$ be a group, $V$ be a finite-dimensional irreducible
complex $G$-module, and let
\[
  0\ne f\in\Sym^d(V^*),
  \qquad d\geq3,
\]
have $G$-invariant Jacobian $J(f)$.
With the notation above, we have:
\begin{enumerate}[label=\textup{(\roman*)},leftmargin=0.9cm]

\item $\cF_f^{(d)} =
  \left\{
    \sum_{i=1}^r c_i f_i:
    c_i\in\C^\times
  \right\}$.

\smallskip

\item
The stabilizer $K$ preserves the line $\C f_1$. Moreover, if
\[
  k\cdot f_1=\psi(k)f_1
  \qquad(k\in K)
\]
for the corresponding linear character $\psi$ of $K$, then
\[
  J(f)\cong\Ind_K^G(\psi\otimes W).
\]

\smallskip

\item
The fibre $\cF_f^{(d)}$ contains a relative invariant if and only if
$\psi$ extends to a character $\chi$ of $G$.  If
$t_1=1,\ldots,t_r$ are representatives of $G/K$, then one such form
is
\[
  f_\chi
  =
  \sum_{i=1}^r
  \chi(t_i)^{-1}t_i\cdot f_1.
\]

\end{enumerate}
\end{theorem}

\begin{proof}
Part~\textup{(i)} follows immediately from
\eqref{eq:homogeneous-fibre} and $\rad Z(f)=0$.    Now let us prove  \textup{(ii)}. For any $k\in K$, we have $kV_1=V_1$.  Since $kV_1=\alpha_k(e_1)V$ (cf. \eqref{eq:block-action}) 
 and $\alpha_k$ permutes the primitive idempotents of $Z(f)$, it
follows that $\alpha_k(e_1)=e_1$.  Write
\[
  a_k=\sum_{i=1}^r\lambda_i(k)e_i.
\]
Using $k\cdot f=f_{a_k}$ and \eqref{eq:summand-action}, $k\cdot f_1
  =
  (k\cdot f)_{e_1}
  =
  (f_{a_k})_{e_1}
  =
  \lambda_1(k)f_1$. Thus $\C f_1$ is $K$-invariant. Let $\psi$ be the corresponding
character. Since  $f$ is nondegenerate, and $f=f_1+\cdots +f_r$ is a decomposition into a sum of forms in disjoint sets of variables, we know each $f_i$ is nondegenerate on $V_i$. By Lemma \ref{lem:Xi-differentiation}, $J(f_1)\cong\psi\otimes W$ as $K$-modules. Choose representatives $t_1=1,\ldots,t_r$ of $G/K$ with
$t_iW=V_i$.  By \eqref{eq:summand-action}, $t_i\cdot f_1=\mu_i f_i$. Hence
\[
  t_i\cdot J(f_1)
  =
  J(t_i\cdot f_1)
  =
  J(f_i).
\]
Since $J(f)=\bigoplus_{i=1}^rJ(f_i)$, we obtain
\[
  J(f)
  =
  \bigoplus_{i=1}^r t_i\cdot J(f_1)
  \cong
  \Ind_K^G J(f_1)
  \cong
  \Ind_K^G(\psi\otimes W).
\]
Finally, suppose that $h=\sum_{i=1}^r c_i f_i\in\cF_f^{(d)}$ is a relative invariant corresponding to $\chi$.  For $k\in K$, the block $V_1=W$ is fixed, so
comparison of the $V_1$-components in  $k\cdot h=\chi(k)h$ gives
\[
  c_1\psi(k)f_1=c_1\chi(k)f_1.
\]
Since $c_1\ne0$, we obtain $\psi=\chi|_K$. Conversely, suppose that $\chi|_K=\psi$, and set
\[
  f_\chi
  =
  \sum_{i=1}^r
  \chi(t_i)^{-1}t_i\cdot f_1.
\]
For $g\in G$, write
\[
  gt_i=t_{\sigma(i)}k_i,
  \qquad k_i\in K,
\]
where $\sigma$ is the permutation of $G/K$ induced by $g$.  Then
\[
\begin{aligned}
  g\cdot f_\chi
  &=
  \sum_i
  \chi(t_i)^{-1}t_{\sigma(i)}
  \cdot(k_i\cdot f_1)\\
  &=
  \sum_i
  \chi(t_i)^{-1}\chi(k_i)\,
  t_{\sigma(i)}\cdot f_1\\
  &=
  \chi(g)
  \sum_i
  \chi(t_{\sigma(i)})^{-1}
  t_{\sigma(i)}\cdot f_1\\
  &=
  \chi(g)f_\chi.
\end{aligned}
\]
Thus $f_\chi$ is a relative invariant corresponding to $\chi$.  Moreover,
$t_i\cdot f_1$ is a nonzero scalar multiple of $f_i$,
so all its components are nonzero.  Hence, by
\textup{(i)}, $f_\chi\in\cF_f^{(d)}$.
\end{proof}

\begin{remark}\label{rem:computability}
The objects in Theorem~\ref{thm:jacobian-structure} can be computed
explicitly.  Following \cite[Algorithm~3.12]{HLYZ}, one first computes
the center $Z(f)$ by solving the linear equations
\[
  X^{\mathsf T}H_f=H_fX.
\]
Since $Z(f)$ is semisimple in the present setting, its primitive
orthogonal idempotents can then be obtained by simultaneous
diagonalization, and these determine the canonical decomposition
\[
  V=\bigoplus_i V_i,\qquad f=\sum_i f_i.
\]
The action of $G$ on the
blocks then determines $K=\Stab_G(V_1)$ and the character $\psi$.
\end{remark}

Theorem~\ref{thm:primitive} is now an immediate consequence of Theorem~\ref{thm:jacobian-structure}.

\begin{proof}[Proof of Theorem~\ref{thm:primitive}]
By Theorem~\ref{thm:jacobian-structure}, the canonical blocks of $f$
are permuted transitively by $G$.  If there were more than one block,
they would form a nontrivial system of imprimitivity for $V$.
Since $V$ is primitive, there is therefore only one canonical block.
Thus
\[
  W=V,\qquad K=G,\qquad f_1=f.
\]
Part~\textup{(ii)} of
Theorem~\ref{thm:jacobian-structure} then shows that $\C f$ is
$G$-invariant.
\end{proof}

\begin{corollary}\label{cor:connected-semisimple}
Let $G$ be a connected semisimple complex algebraic group and let $V$
be an irreducible rational $G$-module.  If
$0\ne f\in\Sym^d(V^*)$, $d\geq3$, and $J(f)$ is $G$-invariant, then
$f$ is fixed by $G$.
\end{corollary}

\begin{proof}
Suppose that $V$ admitted a nontrivial system of imprimitivity $V=V_1\oplus\cdots\oplus V_r$. For each $i$, the orbit map
\[
  G\longrightarrow \mathrm{Gr}(\dim V_i,V),
  \qquad
  g\longmapsto gV_i,
\]
is a morphism whose image is contained in the finite set
$\{V_1,\ldots,V_r\}$.  Since $G$ is connected, its image is a
single point.  Hence every $V_i$ is $G$-stable, contradicting the
irreducibility of $V$ when $r>1$.  Thus $V$ is primitive. By Theorem~\ref{thm:primitive}, there is a linear character $\chi$
such that
\[
  g\cdot f=\chi(g)f.
\]
Since $\C f$ is a one-dimensional submodule of the rational
$G$-module $\Sym^d(V^*)$, the character $\chi$ is algebraic.
A connected semisimple complex algebraic group has no nontrivial
algebraic characters.  Hence $\chi=1$, and $f$ is fixed by $G$.
\end{proof}

This recovers the irreducible case of
\cite[Theorem~1]{Kempf}.  In the present formulation, the two
algebraic-group hypotheses have separate roles: connectedness forces
an irreducible rational module to be primitive, while semisimplicity
eliminates the remaining character.

\subsection{Primitive modules for finite groups}

Theorem~\ref{thm:jacobian-structure} and
Theorem~\ref{thm:primitive} hold for arbitrary groups.  We now assume
that $G$ is finite and show that primitivity
exactly characterizes the finite-dimensional irreducible modules for which
\eqref{eq:main-question} holds. The key point is that every nonzero finite-dimensional module for a
finite group admits a nondegenerate invariant form in some degree.

\begin{lemma}\label{lem:finite-nondegenerate-invariant}
Let $K$ be a finite group and $W$ be a  finite-dimensional  complex $K$-module. Then $\Sym^d(W^*)$ contains a nondegenerate $K$-invariant form for some $d\geq3$.
\end{lemma}

\begin{proof}
Choose linear forms $\ell_1,\ldots,\ell_m$ spanning $W^*$ and set
\[
  q=\prod_{j=1}^m\prod_{k\in K} k\cdot\ell_j.
\]
The group $K$ permutes the factors, so $q$ is $K$-invariant.  Its
linear factors span $W^*$, since the original
$\ell_1,\ldots,\ell_m$ occur among them. Suppose that $D_vq=0$.  Let $L$ be any linear factor of $q$, and
write
\[
  q=L^a H,
  \qquad a\geq1,\qquad L\nmid H.
\]
Then $0=D_vq
   =L^{a-1}\bigl(aL(v)H+L D_v H\bigr)$, and hence
\[
  aL(v)H+L D_v H=0.
\]
If $L(v)\ne0$, this identity would imply $L\mid H$, a contradiction.
Thus $L(v)=0$ for every linear factor $L$ of $q$.  Since these
factors span $W^*$, it follows that $v=0$.  Hence $q$ is
nondegenerate. If $\deg q<3$, replace $q$ by a sufficiently large positive power $q^m$.
Since
\[
  D_v(q^m)=m q^{m-1}D_vq,
\]
$q^m$ is also nondegenerate.  This completes the proof.
\end{proof}

\begin{proof}[Proof of Theorem~\ref{thm:finite-converse}]
If $V$ is primitive, Theorem~\ref{thm:primitive} gives the required
implication in every degree. Conversely, suppose that $V$ is imprimitive.  Choose a block $W$ and
put $K=\Stab_G(W)$. Then
\[
  V=\bigoplus_{t\in G/K}tW
  \cong\Ind_K^G W.
\]
By Lemma~\ref{lem:finite-nondegenerate-invariant}, for some $d\geq3$, 
there is a nondegenerate $K$-invariant form $q\in\Sym^d(W^*)$. Choose representatives $t_1=1,t_2,\ldots,t_r$ for $G/K$ and put  $W_i=t_iW$. Transport $q$ to a form $q_i\in\Sym^d(W_i^*)$ by
\[
  q_i(w)=q(t_i^{-1}w)
  \qquad(w\in W_i),
\]
and regard $q_i$ as a form on $V$ depending only on
the $W_i$-variables.   For $c_1,\ldots,c_r\in\C^\times$, set
\begin{equation*}
  f_c=\sum_{i=1}^r c_iq_i.
\end{equation*}
Since the $q_i$ depend on disjoint blocks,
\begin{equation}\label{eq:constructed-J}
  J(f_c)=\bigoplus_{i=1}^r J(q_i).
\end{equation}
We claim that $ J(f_c)$ is $G$-invariant.  If $gW_i=W_j$, then $gt_i=t_jk$ for some $k\in K$.  Since $q$ is $K$-invariant,
\[
  g\cdot q_i
  =
  t_j\cdot(k\cdot q)
  =
  q_j.
\]
Hence, by covariance of the Jacobian,
\[
  g\cdot J(q_i)
  =
  J(g\cdot q_i)
  =
  J(q_j).
\]
Thus $G$ permutes the summands in
\eqref{eq:constructed-J}, and $J(f_c)$ is $G$-invariant.

It remains to choose the coefficients so that $\C f_c$ is not
$G$-invariant.   Since $r>1$ and $G$ acts
transitively on $\{q_1,\dots,q_r\}$, some $g\in G$ induces a nontrivial
permutation $\sigma$ of $\{1,\ldots,r\}$, such that $g\cdot q_i=q_{\sigma(i)}$. Let $c_i=2^i\ (1\leq i\leq r)$. Suppose that $\C f_c$ were $G$-invariant.  Then
\[
  g\cdot f_c=\lambda f_c
\]
for some $\lambda\in\C^\times$.  Since $q_1,\dots,q_r$ are linearly independent,
comparison of coefficients gives
\[
  c_i=\lambda c_{\sigma(i)},\qquad 1\leq i\leq r.
\]
 Since all $c_i$ are positive, $\lambda>0$.
If $m$ is the order of $\sigma$, iteration gives $\lambda^m=1$,
hence $\lambda=1$ and $\sigma$ is trivial, which is a contradiction.  Therefore $\C f_c$ is not
$G$-invariant.
\end{proof}

\section{Invariant Jacobians for Symmetric Groups}
\label{sec:Sn}

We apply the preceding results to symmetric groups. For irreducible Specht modules, the classification of systems of imprimitivity due to Djoković--Malzan  \cite{DjokovicMalzan} makes Theorem 3.1 explicit. We then consider the natural permutation module \(P_n=\mathbb C^n\).  On $P_n$, invariance of $\C f$ is
equivalent to $f$ being symmetric or alternating, whereas  invariance
of $J(f)$  gives a broader notion of
symmetry for homogeneous polynomials. Since \(P_n\) is reducible, Theorem 3.1 does not apply directly. We prove that the center of a nondegenerate form is still semisimple, and then determine all possible canonical decompositions, yielding a
classification of all forms on $P_n$ with invariant Jacobian.

\subsection{Irreducible modules}

Assume throughout this section that $n\geq5$.  The irreducible
complex $S_n$-modules are the Specht modules $S^\lambda$, indexed by
partitions $\lambda\vdash n$.  We write $\lambda'$ for the conjugate
partition.   We use the classification of imprimitive irreducible characters of
$S_n$ due to Djokovi\'c--Malzan
\cite[Theorem~1]{DjokovicMalzan}, together with the standard
restriction theorem from $S_n$ to $A_n$,  see \cite[Theorem~2.5.7]{JamesKerber} or  \cite[Section~5.1]{FultonHarris}.

\begin{theorem}\label{thm:DM}
Let $n\geq5$. 
\begin{enumerate}[label=\textup{(\roman*)},leftmargin=0.9cm]
\item The irreducible $S_n$-module $S^\lambda$ is primitive
if and only if $\lambda\ne\lambda'$. 
\smallskip

\item If $\lambda=\lambda'$, then
\begin{equation}\label{eq:An-restriction}
  \Res_{A_n}^{S_n}S^\lambda
  =
  W_\lambda^+\oplus W_\lambda^-,
\end{equation}
where $W_\lambda^+$ and $W_\lambda^-$ are nonisomorphic irreducible
$A_n$-modules.  Every odd permutation interchanges the two summands,
so \eqref{eq:An-restriction} gives a two-block system of imprimitivity.
\end{enumerate}
\end{theorem}

The preceding classification, together with
Theorem~\ref{thm:jacobian-structure}, gives the following description.

\begin{theorem}
\label{thm:Sn-classification}
Let $V=S^\lambda$, $0\ne f\in\Sym^d(V^*)$, with $d\geq3$, and $J(f)$ is $S_n$-invariant.
\smallskip

\begin{enumerate}[label=\textup{(\roman*)},leftmargin=0.9cm]
\item
If $\lambda\neq \lambda'$,  then 
\[
  \sigma\cdot f=\chi(\sigma)f
  \qquad(\sigma\in S_n)
\]
for some $\chi\in\{1,\sgn\}$, and
\begin{equation}\label{eq:Sn-one-block}
  J(f)\cong\chi\otimes S^\lambda\cong \begin{cases}
  S^{\lambda} & \text{if}\ \chi=1\\
  S^{\lambda'} & \text{if}\ \chi=\sgn
  \end{cases},
  \qquad
  \cF_f^{(d)}=\C^\times f.
\end{equation}

\item  If $\lambda=\lambda'$ and the canonical imprimitive
decomposition of $f$ is
\[
  \Res_{A_n}^{S_n}S^\lambda
  =
  W_\lambda^+\oplus W_\lambda^-.
\]
Let $e_+,e_-$ be the corresponding primitive idempotents of $Z(f)$,
and put $p=f_{e_+}$. Fix a transposition $\tau$, and put $q=\tau\cdot p$. Then
\begin{equation}\label{eq:Sn-two-block}
  J(f)\cong S^\lambda,
  \qquad
  \cF_f^{(d)}
  =
  \{\,a p+b q:\ a,b\in\C^\times\,\}.
\end{equation}
  Moreover, the fibre contains the invariant form $p+q$
 and the sign-relative form $p-q$.
\end{enumerate}
\end{theorem}

\begin{proof}
If $\lambda\neq \lambda'$, $S^{\lambda}$ is primitive, Theorem \ref{thm:primitive} shows that $\C f$ is $S_n$-invariant. Since the only linear characters of $S_n$ are $1$ and $\sgn$, Theorem \ref{thm:jacobian-structure} gives
\eqref{eq:Sn-one-block}.  Now suppose $\lambda=\lambda'$, and that the canonical decomposition is
\[
  \Res_{A_n}^{S_n}S^\lambda
  =
  W_\lambda^+\oplus W_\lambda^-.
\]
Since $A_n$ is perfect for $n\geq5$, it has no nontrivial linear
characters.  By
Theorem~\ref{thm:jacobian-structure}, $k\cdot p=p$ for any $k\in A_n$ and
\[
  J(f)
  \cong
  \Ind_{A_n}^{S_n}W_\lambda^+
  \cong
  S^\lambda.
\]
Since $\tau$
interchanges $W_\lambda^+$ and $W_\lambda^-$, we know $q=\tau\cdot p=\lambda f_{e_-}$, with $\lambda\neq 0$.  Thus Theorem~\ref{thm:jacobian-structure}\textup{(i)} gives
\[
  \cF_f^{(d)}
  =
  \{\,ap+bq:a,b\in\C^\times\,\}.
\]
Moreover, for $k\in A_n$,
\[
  k\cdot q
  =
  k\tau\cdot p
  =
  \tau(\tau^{-1}k\tau)\cdot p
  =
  \tau\cdot p
  =
  q,
\]
thus $A_n$ fixes both $p$ and $q$.
On the other hand,
\[
  \tau\cdot p=q,
  \qquad
  \tau\cdot q=p.
\]
Since $S_n$ is generated by $A_n$ and $\tau$, it follows that $p+q$ is $S_n$-invariant, while $p-q$ affords the sign character.
\end{proof}

\begin{example}\label{cor:standard}
Consider the standard irreducible module of $S_n$, which is
\[
  H_n
  =
  \{(x_1,\ldots,x_n)\in\C^n:
    x_1+\cdots+x_n=0\}
  \cong
  S^{(n-1,1)}.
\]
For $n\geq5$, the partition $(n-1,1)$ is not self-conjugate, so
$H_n$ is primitive by Theorem~\ref{thm:DM}.   Thus for $0\ne f\in\Sym^d(H_n^*)$, with $d\geq 3$, the $S_n$-invariance of $J(f)$ implies that $\C f$ is invariant. Here, the restriction $n\geq5$ is essential.  Indeed, the standard modules of $S_3$ and $S_4$ are imprimitive. For example, consider
\[
  H_3=\{(x_1,x_2,x_3)\in\C^3:x_1+x_2+x_3=0\},
\]
and let $\omega$ be a primitive cube root of unity.  Put
\[
  \ell_1=x_1+\omega x_2+\omega^2x_3,
  \qquad
  \ell_2=x_1+\omega^2x_2+\omega x_3.
\]
The two lines $\C\ell_1$ and $\C\ell_2$ are permuted by $S_3$, thus $H_3^*=\C\ell_1\oplus \C\ell_2$ is a system of imprimitivity for
$H_3^*$. Since $H^*_3\simeq H_3$, we know $H_3$ is not primitive.  Moreover, consider
\[
  f=\ell_1^4+2\ell_2^4.
\]
We have $J(f)=\langle \ell_1^3,\ell_2^3\rangle$, which is $S_3$-invariant.  On the other hand,
\[
  (23)\cdot f=\ell_2^4+2\ell_1^4
\]
is not a scalar multiple of $f$.  Thus $\C f$ is not
$S_3$-invariant.
\end{example}

\subsection{The natural permutation module}
\label{sec:permutation}

We still assume $n\geq 5$. Let $S_n$ act by coordinate permutations on
\begin{equation}\label{eq:naturalrepsn}
  P_n=\C^n=T\oplus H_n,
  \qquad
  \text{in which}\quad T=\C t,\quad t=(1,\ldots,1).
\end{equation}
The module $P_n$ is reducible, so Lemma~\ref{lem:radical-proper} does not
immediately imply that $\rad Z(f)=0$.  However, the decomposition $P_n=T\oplus H_n$ leaves only a few possibilities.

\begin{lemma}\label{lem:permutation-radical}
Let $d\geq3$, 
$0\ne f\in\Sym^d(P_n^*)$ be nondegenerate.  If $J(f)$ is
$S_n$-invariant, then
\[
  \rad Z(f)=0.
\]
\end{lemma}

\begin{proof}
Put $N=\rad Z(f)$, $U=NP_n$. By Lemma~\ref{lem:radical-proper}, $U$ is a proper
$S_n$-submodule of $P_n$.  Since $P_n=T\oplus H_n$,  the only possibilities are
\[
  U=0,\qquad U=T,\qquad\text{or}\qquad U=H_n.
\]
If $U=0$, then every element of $N$ acts trivially on $P_n$, and
hence $N=0$.  We show that the other two possibilities lead to
contradictions. Recall that $S_n$ normalizes $Z(f)$ and hence also its Jacobson
radical $N$.  For convenience, write
\[
  B_f(u,v)=D_uD_vf.
\]
The defining identity of the center is
\begin{equation}\label{eq:center-second-derivative}
  B_f(Au,v)=B_f(u,Av)
  \qquad(A\in Z(f),\ u,v\in P_n).
\end{equation}

Suppose first that $U=T$.  Then every $A\in N$ has image contained
in $T$.  Since $A$ is nilpotent and $T$ is one-dimensional, $A$
vanishes on $T$.   Thus restriction to $H_n$ gives an $S_n$-equivariant injection
\[
  \rho:N\hookrightarrow\Hom(H_n,T),
  \qquad
  \rho(A)=A|_{H_n}.
\]
Since $N\ne0$ and $\Hom(H_n,T)\cong H_n^*$ is irreducible, we have $\rho(N)=\Hom(H_n,T)$.  Since $T=\C t$, where $t$ is defined in \eqref{eq:naturalrepsn},  we further identify
\[
  H_n^*\cong\Hom(H_n,T),
  \qquad
  \varphi\longmapsto
  \bigl(v\mapsto\varphi(v)t\bigr).
\]
Therefore, for every $\varphi\in H_n^*$, the corresponding element
of $N$ is the unique endomorphism $A_\varphi$ of
$P_n=T\oplus H_n$ satisfying
\[
  A_\varphi(t)=0,
  \qquad
  A_\varphi(v)=\varphi(v)t
  \quad(v\in H_n).
\]
Substituting $A_\varphi$ into \eqref{eq:center-second-derivative} gives, for any $u,v\in H_n$, 
\begin{equation}\label{eq:keybilinear-1}
  \varphi(u)B_f(t,v)
  =
  \varphi(v)B_f(u,t).
\end{equation}
For each $0\ne v\in H_n$, since $\dim H_n\geq 2$, we can choose $u$ linearly independent from $v$ and
$\varphi\in H_n^*$ such that
\[
  \varphi(u)=1,
  \qquad
  \varphi(v)=0.
\]
Thus \eqref{eq:keybilinear-1} shows $B_f(t,v)=0$ for any $v\in H_n$. Now choose $v\in H_n$ and $\varphi\in H_n^*$ with
$\varphi(v)\ne0$.  Applying
\eqref{eq:center-second-derivative} to the pair $(t,v)$ gives
\[
  0
  =
  B_f(A_\varphi t,v)
  =
  B_f(t,A_\varphi v)
  =
  \varphi(v)B_f(t,t),
\]
thus $B_f(t,t)=0$. Since $P_n=T\oplus H_n$, we conclude that  $B_f(t,v)=0$ for any $v\in P_n$. Equivalently,
\[
  D_v(D_tf)=0
  \qquad \text{for any}\ v\in P_n.
\]
Thus every directional derivative of $D_tf$ vanishes.  Since
$D_tf$ is homogeneous of positive degree, $D_tf=0$, contradicting
the nondegeneracy of $f$.  Therefore $U=T$ is impossible.

Suppose next that $U=H_n$. Then every $A\in N$ has image contained in $H_n$.  We first show
that every element of $N$ vanishes on $H_n$. Consider the common kernel
\[
  K_0=
  \{v\in H_n:Av=0\text{ for all }A\in N\}.
\]
It is nonzero.  Indeed, if $m$ is minimal such that $N^mH_n=0$, then
\[
  0\ne N^{m-1}H_n\subseteq K_0.
\]
Moreover, $K_0$ is $S_n$-stable: if $v\in K_0$, $\sigma\in S_n$,
and $A\in N$, then
\[
  A(\sigma v)
  =
  \sigma\bigl((\sigma^{-1}A\sigma)v\bigr)
  =
  0,
\]
since $N$ is stable under conjugation.  By irreducibility of $H_n$,
we obtain $K_0=H_n$. Hence
\[
  A(H_n)=0
  \qquad(A\in N).
\]
Thus every $A\in N$ satisfies
\[
  A(H_n)=0,
  \qquad
  A(T)\subseteq H_n.
\]
Consequently, restriction to $T$ gives an $S_n$-equivariant
injection
\[
  \rho:N\hookrightarrow\Hom(T,H_n),
  \qquad
  \rho(A)=A|_T.
\]
Since $T=\C t$ is the trivial $S_n$-module, evaluation at $t$ gives
an $S_n$-module isomorphism
\[
  \Hom(T,H_n)\cong H_n,
  \qquad
  \Phi\longmapsto\Phi(t).
\]
If $N\ne0$, irreducibility of $H_n$ therefore gives $ \rho(N)=\Hom(T,H_n)$. Under this identification, any vector $w\in H_n$ corresponds to
the homomorphism
\[
  \Phi_w:T\longrightarrow H_n,
  \qquad
  \Phi_w(t)=w.
\]
Since $\rho$ is an isomorphism onto $\Hom(T,H_n)$, there is a unique
element of $N$ whose restriction to $T$ is $\Phi_w$.  As every
element of $N$ vanishes on $H_n$, this element is precisely the
endomorphism $A_w$ of $P_n=T\oplus H_n$ defined by
\[
  A_w(t)=w,
  \qquad
  A_w(v)=0
  \quad(v\in H_n).
\]
Now take any $v\in H_n$.  Applying
\eqref{eq:center-second-derivative} to $A_w$ and the pair $(t,v)$
gives
\[
  B_f(w,v)
  =
  B_f(A_wt,v)
  =
  B_f(t,A_wv)
  =
  0.
\]
Hence  $B_f(H_n,H_n)=0$. Fix $w\in H_n$.  Then
\[
  D_v(D_wf)=0
  \qquad(v\in H_n),
\]
so $D_wf$ is constant along the $H_n$-directions and therefore
depends only on the one-dimensional quotient $P_n/H_n$.  Since the
annihilator of $H_n$ in $P_n^*$ is spanned by
\[
  s=x_1+\cdots+x_n,
\]
and $D_wf$ is homogeneous of degree $d-1$, we have $ D_wf=c_ws^{d-1}$. Therefore the image of
\[
  H_n\longrightarrow J(f),
  \qquad
  w\longmapsto D_wf,
\]
is contained in the one-dimensional space $\C s^{d-1}$.  Thus the map
\[
  P_n\longrightarrow J(f),
  \qquad
  v\longmapsto D_vf,
\]
is not an isomorphism, since $\dim H_n\geq 2$. This  contradicts  the assumption that $f$ is nondegenerate, see \eqref{eq:nondegeneracyiso}. Thus  $U=H_n$ is also impossible.
\end{proof}

Recall that the primitive orthogonal idempotents $e_1,\dots,e_r$ of $Z(f)$ determine the canonical decomposition
\[
  P_n=V_1\oplus\cdots\oplus V_r,
  \qquad
  V_i=e_iP_n,
\]
and  $S_n$ permutes the primitive idempotents as well as the canonical blocks $V_i$.

\begin{lemma}\label{lem:permutation-blocks}
Let $f$ be as in Lemma~\ref{lem:permutation-radical}.  Its canonical
decomposition has exactly one of the following forms:
\begin{enumerate}[label=\textup{(\roman*)},leftmargin=0.9cm]
\item one block $P_n$;
\item the two blocks $T$ and $H_n$;
\item $n$ one-dimensional blocks permuted transitively by $S_n$.
\end{enumerate}
\end{lemma}

\begin{proof}
Let $V_1,\ldots,V_r$ be the canonical blocks.  The sum of the blocks
in each $S_n$-orbit is an $S_n$-submodule of $P_n$.   Since $P_n=T\oplus H_n$, there are at most two orbits. If there are two orbits, their sums must be $T$ and $H_n$.  The
one-dimensional space $T$ is one block, and the primitivity of $H_n$
from Example~\ref{cor:standard} implies that $H_n$ is also one
block.  This gives \textup{(ii)}. If there is only one orbit, then $S_n$ acts transitively on 
$V_1,\ldots,V_r$ and $e_1,\ldots,e_r$.  Recall the element $t=(1,\dots,1)\in P_n$, and write
\[
  t=t_1+\cdots+t_r,
  \qquad
  t_i=e_it\in V_i.
\]
Since $t$ is fixed by $S_n$ and the $e_i$ are permuted,
the vectors $t_i$ are permuted in the same way. Hence $  L=\langle t_1,\ldots,t_r\rangle$ is an $S_n$-submodule containing $T$.  As $ P_n=\bigoplus_{i=1}^r V_i$, the vectors $t_1,\ldots,t_r$ are linearly independent.  If $r=1$, we obtain \textup{(i)}.  If $r>1$, then $L=P_n$, and so $r=n$, we obtain \textup{(iii)}.
\end{proof}

Building on Lemmas \ref{lem:permutation-radical}--\ref{lem:permutation-blocks}, we are now in a position to answer Xi's question completely for the
natural permutation representation $P_n$. 

\begin{theorem}
\label{thm:permutation-classification}
Let $n\geq5$, $d\geq3$, and
\[
 0\neq f\in\Sym^d(P_n^*)= \C[x_1,\ldots,x_n]_d.
\]
Then $J(f)$ is $S_n$-invariant if and only if $f$ belongs
to one of the following families:
\begin{enumerate}[label=\textup{(\roman*)},leftmargin=0.9cm]
\item
The line $\C f$ is $S_n$-invariant, that is, $f$ is symmetric
or alternating.

\item $f=a s^d+h$, where $s=x_1+\cdots+x_n$,  $a\in\C$ and $h$ is the inflation to $P_n$ of a degree-$d$
form on $H_n$ which is either $S_n$-invariant or
$\sgn$-relative invariant.

\item  $f=\sum_{i=1}^n c_i\ell_i^d$, where $c_i\in\C^\times$, and
\begin{equation}\label{eq:permutation-linear-forms}
  \ell_i=x_i+\gamma s,
  \qquad
  1+n\gamma\ne0, \gamma\in \C.
\end{equation}
\end{enumerate}
\end{theorem}

\begin{proof}
We first prove sufficiency.  In \textup{(i)}, invariance of $\C f$
implies invariance of $J(f)$ by \eqref{eq:J-covariance}.  In
\textup{(ii)}, $J(f)=\langle a s^{d-1}\rangle+J(h)$, and both terms are $S_n$-invariant. In \textup{(iii)}, since 
\[D_vf
  =
  d\sum_{i=1}^n
  c_i\ell_i(v)\ell_i^{d-1},\]
we know $J(f)\subset 
  \langle
    \ell_1^{d-1},\ldots,\ell_n^{d-1}
  \rangle$. Define
\[
  \Phi_\gamma:P_n^*\longrightarrow P_n^*,
  \qquad
  \Phi_\gamma(x_i)=\ell_i=x_i+\gamma s.
\]
 $\Phi_\gamma$ is identity on $H_n^*$, while $ \Phi_\gamma(s)=(1+n\gamma)s$. Hence the condition $1+n\gamma\ne0$ implies that $\Phi_\gamma$ is invertible and $\ell_1,\dots,\ell_n$ form a basis of $P_n^*$.  Then it follows that
\[
  J(f)
  =
  \langle
    \ell_1^{d-1},\ldots,\ell_n^{d-1}
  \rangle.
\]
Since $S_n$ permutes the $\ell_i$, the space $J(f)$ is invariant.

Conversely, suppose that $J(f)$ is $S_n$-invariant.  Then
$f^\perp$ is an $S_n$-submodule of $P_n$, thus
\[
  f^\perp\in\{0,T,H_n,P_n\}.
\]
If $f^\perp=P_n$, then $f=0$.  If $f^\perp=H_n$, then $f$ is constant on the cosets of $H_n$ and hence factors through
the quotient $P_n/H_n\cong T$. Thus $f$ depends only on $s$ and  $f=a s^d$. If
$f^\perp=T$, then $f$ is the inflation of a nondegenerate form on
$H_n$ with invariant Jacobian. These cases are contained in
\textup{(ii)}. Finally, let $f^\perp=0$, that is $f$ is nondegenerate.  By
Lemmas~\ref{lem:permutation-radical}--\ref{lem:permutation-blocks}, $Z(f)$ is semisimple, and there are three possibilities:
\begin{itemize}[leftmargin=0.7cm]
\item If the canonical decomposition has one block, then
$Z(f)=\C\operatorname{id}_{P_n}$.  Hence  $\C f$ is $S_n$-invariant, this
gives \textup{(i)}.

\item If the canonical blocks are $T$ and $H_n$, then
\[
  f=a s^d+h,
  \qquad a\ne0,
\]
with $h$ nondegenerate on $H_n$.  Both blocks are fixed by $S_n$.  Let
$e_T$ and $e_H$ be the corresponding primitive idempotents, then  $Z(f)=\C e_T\oplus\C e_H$. Since $J(f)$ is $S_n$-invariant, for every $\sigma\in S_n$ there is
a unit $a_\sigma\in Z(f)^\times$ such that $\sigma\cdot f=f_{a_\sigma}$, cf. \eqref{eq:unit-ag}. Thus
\[
  a_\sigma
  =
  \lambda_T(\sigma)e_T+\lambda_H(\sigma)e_H
\]
for some $\lambda_T(\sigma),\lambda_H(\sigma)\in\C^\times$, and
therefore
\[
  f_{a_\sigma}
  =
  \lambda_T(\sigma)f_T+\lambda_H(\sigma)f_H.
\]
Writing  $f=f_T+f_H=a s^d+h$, and comparing the $H_n$-components, we obtain
\[
  \sigma\cdot h=\lambda_H(\sigma)h.
\]
Hence $\C h$ is $S_n$-invariant. This gives \textup{(ii)}.

\item If the canonical decomposition has $n$ blocks.
Recall from the proof of Lemma~\ref{lem:permutation-blocks} that
\[
  t_i=e_it\in V_i,\qquad i=1,\ldots,n,
\]
form a basis of $P_n$, and that $S_n$ permutes them.  Write
\[
  t_1=(a_1,\ldots,a_n).
\]
If the
distinct coordinate values of $t_1$ occur with multiplicities
$m_1,\ldots,m_r$, then
\[
  \lvert S_n\cdot t_1\rvert
  =
  \frac{n!}{m_1!\cdots m_r!}
  =
  n.
\]
Hence $m_1!\cdots m_r!=(n-1)!$. For $n\ge5$, this forces the multiplicities to be $n-1$ and $1$. Thus, after relabelling the standard coordinates, there exist
$\alpha,\beta\in\C$, with $\alpha\ne0$, such that
\[
  t_i=\alpha u_i+\beta t,
  \qquad i=1,\ldots,n,
\]
where $u_1,\ldots,u_n$ is the standard basis of $P_n$.  Since
\[
  t_1+\cdots+t_n=t,
\]
we obtain $\alpha+n\beta=1$. Let $\lambda_1,\ldots,\lambda_n$ be the basis of $P_n^*$ dual to
$t_1,\ldots,t_n$,  then
\[
  \lambda_i
  =
  \frac{1}{\alpha}(x_i-\beta s).
\]
Set $\gamma=-\beta$, $\ell_i=x_i+\gamma s$. Since $\alpha+n\beta=1$,  $1+n\gamma=\alpha\neq 0$. Each canonical block $V_i=\C t_i$ is one-dimensional, so the
corresponding canonical summand of $f$ is a scalar multiple of
$\lambda_i^d$, and hence also of $\ell_i^d$.  Thus $f=\sum_{i=1}^n c_i\ell_i^d$,  and nondegeneracy implies $c_i\ne0\ (1\le i\le n)$, this gives \textup{(iii)}.
\end{itemize}
This completes the proof.
\end{proof}

Theorem~\ref{thm:permutation-classification} classifies all forms
with $S_n$-invariant Jacobian.  Together with
Proposition~\ref{prop:Jacobian-fibre} and
Lemma~\ref{lem:Xi-differentiation},  the structure of the
Jacobian module and the Jacobian fibre can be described as follows:

\begin{corollary}
\label{cor:permutation-Xi}
Assume that $f$ in
Theorem~\ref{thm:permutation-classification} is nondegenerate.

\begin{enumerate}[label=\textup{(\roman*)},leftmargin=0.9cm]

\item
In the one-block case, suppose that $\C f$ affords the character
$\chi\in\{1,\sgn\}$.  Then
\[
  J(f)\cong\chi\otimes P_n,
  \qquad
  \cF_f^{(d)}=\C^\times f.
\]

\item
In the two-block case, write $f=a s^d+h$, where $\C h$ affords the character
$\chi\in\{1,\sgn\}$.  Then
\[
  J(f)
  \cong
  \Triv\oplus(\chi\otimes H_n),
  \qquad
  \cF_f^{(d)}
  =
  \{\,b s^d+c h:b,c\in\C^\times\,\}.
\]
Here $\Triv$ denotes the trivial one-dimensional module.  Moreover,
the fibre contains a relative invariant if and only if $\chi=1$.

\item
In the $n$-block case, write $f=\sum_{i=1}^n c_i\ell_i^d$, as in Theorem~\ref{thm:permutation-classification}\textup{(iii)}. Then
\[
  J(f)\cong P_n,
  \qquad
  \cF_f^{(d)}
  =
  \left\{
    \sum_{i=1}^n b_i\ell_i^d:
    b_i\in\C^\times
  \right\}.
\]
In particular,  $\sum_{i=1}^n\ell_i^d$ is $S_n$-invariant.
\end{enumerate}
\end{corollary}

\bigskip
\noindent
\textbf{Acknowledgment.}
 This work is supported by the National Natural 
Science Foundation of China (No. 12401028).


\begin{thebibliography}{99}

\bibitem{DjokovicMalzan}
D.~\v{Z}. Djokovi\'c and J.~Malzan,
\newblock Imprimitive irreducible complex characters of the symmetric
group,
\newblock \emph{Math. Z.} \textbf{138} (1974), 219--224.

\bibitem{FultonHarris}
W.~Fulton and J.~Harris,
\newblock \emph{Representation Theory: A First Course},
\newblock Graduate Texts in Mathematics, vol.~129,
Springer-Verlag, New York, 1991.

\bibitem{Harrison}
D.~K. Harrison,
\newblock A Grothendieck ring of higher degree forms,
\newblock \emph{J. Algebra} \textbf{35} (1975), 123--138.


\bibitem{HLYZ}
H.-L. Huang, H.~Lu, Y.~Ye, and C.~Zhang,
\newblock On centres and direct sum decompositions of higher degree
forms,
\newblock \emph{Linear Multilinear Algebra} \textbf{70} (2022),
7290--7306.

\bibitem{Hwang}
J.-M. Hwang,
\newblock Symmetrizer group of a projective hypersurface,
\newblock \emph{J. Math. Soc. Japan} \textbf{78} (2026), 55--62.

\bibitem{Isaacs}
I.~M. Isaacs,
\newblock \emph{Character Theory of Finite Groups},
\newblock Academic Press, New York, 1976.

\bibitem{JamesKerber}
G.~James and A.~Kerber,
\newblock \emph{The Representation Theory of the Symmetric Group},
\newblock Encyclopedia of Mathematics and its Applications, vol.~16,
Addison--Wesley, Reading, MA, 1981.

\bibitem{Kempf}
G.~R. Kempf,
\newblock Jacobians and invariants,
\newblock \emph{Invent. Math.} \textbf{112} (1993), 315--321.


\bibitem{MatherYau}
J.~N. Mather and S.~S.-T. Yau,
\newblock Classification of isolated hypersurface singularities by
their moduli algebras,
\newblock \emph{Invent. Math.} \textbf{69} (1982), 243--251.

\bibitem{ORyanShapiro}
M.~O'Ryan and D.~B. Shapiro,
\newblock Centers of higher degree forms,
\newblock \emph{Linear Algebra Appl.} \textbf{371} (2003), 301--314.

\bibitem{Proszynski}
A.~Pr{\'o}szy{\'n}ski,
\newblock On orthogonal decomposition of homogeneous polynomials,
\newblock \emph{Fund. Math.} \textbf{98} (1978), no.~3, 201--217.

\bibitem{Xi}
N.~Xi,
\newblock Module structure on invariant Jacobians,
\newblock \emph{Math. Res. Lett.} \textbf{19} (2012), no.~3,
731--739.

\bibitem{YauMemoir}
S.~S.-T. Yau,
\newblock \emph{Classification of Jacobian Ideals Invariant by
$\mathfrak{sl}(2,\C)$ Actions},
\newblock Mem. Amer. Math. Soc., vol.~72, no.~384, 1988.

\end{thebibliography}
\end{document}